\documentclass[leqno,a4paper,twoside]{article}%
\usepackage{amssymb}
\usepackage{amsfonts}
\usepackage{amsmath}
\usepackage{graphicx}%
\providecommand{\U}[1]{\protect\rule{.1in}{.1in}}
\newtheorem{theorem}{Theorem}

\newtheorem{corollary}[theorem]{Corollary}

\newtheorem{definition}[theorem]{Definition}
\newtheorem{example}[theorem]{Example}

\newtheorem{proposition}[theorem]{Proposition}
\newtheorem{remark}[theorem]{Remark}

\newenvironment{proof}[1][Proof]{\noindent\textbf{#1.} }{\ \rule{0.5em}{0.5em}}
\begin{document}

\title{Nonequivalence of Controllability Properties for Piecewise Linear Markov
Switch Processes}
\author{Dan Goreac\thanks{Universit\'{e} Paris-Est, LAMA (UMR 8050), UPEMLV, UPEC,
CNRS, F-77454, Marne-la-Vall\'{e}e, France, Dan.Goreac@u-pem.fr}\thanks{This
work has been partially supported by he French National Research Agency
project PIECE, number \textbf{{ANR-12-JS01-0006}.}}
\and \date{}}
\maketitle

\begin{center}
{\footnotesize \textbf{Abstract}}
\end{center}

{\footnotesize In this paper we study the exact null-controllability property
for a class of controlled PDMP of switch type with switch-dependent, piecewise
linear dynamics and multiplicative jumps. First, we show that exact
null-controllability induces a controllability metric. This metric is linked
to a class of backward stochastic Riccati equations. Using arguments similar
to the euclidian-valued BSDE in \cite{CFJ_2014}, the equation is shown to be
equivalent to an iterative family of deterministic Riccati equations that are
solvable. Second, we give an example showing that, for switch-dependent
coefficients, exact null-controllability is strictly stronger than approximate
null-controllability. Finally, we show by convenient examples that no
hierarchy holds between approximate (full) controllability and exact
null-controllability. The paper is intended as a complement to
\cite{GoreacGrosuRotenstein_2016} and \cite{GoreacMartinez2015}.\newline}

\begin{center}
{\footnotesize \textbf{R\'{e}sum\'{e}}}
\end{center}

{\footnotesize Nous \'{e}tudions la propri\'{e}t\'{e} de
z\'{e}ro-contr\^{o}labilit\'{e} exacte pour une classe de processus de type
Markovien \`{a} switch ayant des dynamiques lin\'{e}aires par morceaux \`{a}
coefficients switch\'{e}s et bruit multiplicatif. Premi\`{e}rement, nous
montrons que la z\'{e}ro-contr\^{o}labilit\'{e} exacte induit une m\'{e}trique
de contr\^{o}labilit\'{e}. Celle-ci est li\'{e}e \`{a} une classe
d'\'{e}quations stochastiques r\'{e}trogrades de type Riccati. En employant
des arguments similaires aus EDSR classiques (\cite{CFJ_2014}), l'\'{e}quation
de Riccati se r\'{e}duit \`{a} une famille d'\'{e}quations it\'{e}ratives
d\'{e}terministes de type Riccati qui admettent une solution unique.
Deuxi\`{e}mement, nous pr\'{e}sentons un exemple montrant que, pour des
syst\`{e}mes \`{a} coefficients switch\'{e}s, la propri\'{e}t\'{e} de
z\'{e}ro-contr\^{o}labilit\'{e} exacte est strictement plus forte que celle de
z\'{e}ro-contr\^{o}labilit\'{e} approch\'{e}e. Finalement, nous montrons,
\`{a} l'aide d'exemples particuliers, l'impossibilit\'{e} \`{a} \'{e}tablir
une hi\'{e}rarchie entre les propri\'{e}t\'{e}s de contr\^{o}labilit\'{e}
approch\'{e}e (vers une cible arbitraire) et celle de
z\'{e}ro-contr\^{o}labilit\'{e} exacte. Ce travail doit \^{e}tre regard\'{e}
en compl\'{e}ment des \'{e}tudes men\'{e}es dans
\cite{GoreacGrosuRotenstein_2016} et \cite{GoreacMartinez2015}.}

\section{Introduction}

We study the exact null-controllability property for a class of piecewise
deterministic Markov processes of switch type. More precisely, our model
belongs to Markovian systems consisting of a couple mode/ trajectory $\left(
\Gamma,X\right)  .$ The mode $\Gamma$ is a pure jump uncontrolled Markov
process corresponding to spikes inducing regime switching. The second
component $X$ obeys a controlled linear stochastic differential equation (SDE)
with respect to the compensated random measure associated to $\Gamma$. The
linear coefficients governing the dynamics depend on the current mode.

The exact null-controllability problem concerns criteria allowing one to drive
the $X_{T}$ component to zero. This property is of particular importance in
the study of regulatory networks (e.g. \cite{cook_gerber_tapscott_98},
\cite{hasty_pradines_dolnik_collins_00}, \cite{Krishna29032005},
\cite{crudu_debussche_radulescu_09}, etc.) to distinguish, for example, lytic
pathways (e.g. \cite{hasty_pradines_dolnik_collins_00}).

An extensive literature on controllability is available in different
frameworks: finite-dimensional deterministic setting (Kalman's condition,
Hautus test \cite{Hautus}), infinite dimensional settings (via invariance
criteria in \cite{Schmidt_Stern_80}, \cite{Curtain_86},
\cite{russell_Weiss_1994}, \cite{Jacob_Zwart_2001},
\cite{Jacob_Partington_2006}, etc.), Brownian-driven control systems (exact
terminal-controllability in \cite{Peng_94}, approximate controllability in
\cite{Buckdahn_Quincampoix_Tessitore_2006}, \cite{G17}, mean-field
Brownian-driven systems in \cite{G1}, infinite-dimensional setting in
\cite{Fernandez_Cara_Garrido_atienza_99}, \cite{Sarbu_Tessitore_2001},
\cite{Barbu_Rascanu_Tessitore_2003}, \cite{G16}, etc.), jump systems
(\cite{G10}, \cite{GoreacMartinez2015}, etc.). We refer to
\cite{GoreacMartinez2015} for more details on the literature as well as
applications one can address using switch models.

The recent papers \cite{GoreacMartinez2015} and
\cite{GoreacGrosuRotenstein_2016} consider different characterizations of
approximate and approximate null-controllability properties for the same class
of systems. However, they do not address the question of exact
null-controllability (which, in non-stochastic time-homogeneous framework, is
identified with approximate null-controllability).

The aim of the present paper is to offer a complement to the research topics
in \cite{GoreacMartinez2015} and \cite{GoreacGrosuRotenstein_2016}. In
\cite{GoreacMartinez2015}, a Riccati-type argument is used to characterize
approximate controllability for systems with constant coefficients. Our first
aim (in Section \ref{Sec3.1}) is to give an answer to the problem left open in
\cite[Remark 4]{GoreacMartinez2015} where a family of backward stochastic
Riccati equations are presented and absence of results on the solvability is
mentioned. As a by-product, the result provides a metric-type characterization
of exact null-controllability in Section \ref{Sec3.2}. In Section
\ref{Sec3.3}, we give an example of approximate null-controllable system which
fails to be exactly null-controllable. Finally, in Section \ref{Sec3.4}, we
show by convenient examples that no hierarchy can be established between
approximate (full) controllability and exact null-controllability. In other
words, we present an example of approximate controllable yet non exactly
null-controllable system and an example of null-controllable system which
fails to be approximately controllable in certain directions.

We begin with presenting the model and the standing assumptions in Section
\ref{Sec2.1}. The technical constructions allowing to prove the theoretical
results are gathered in Section \ref{SubsTechnicalPrelim}. The controllability
notions (exact null, approximate null, approximate) are given in Section
\ref{Sec2.3}. We gather in the same section some useful results on approximate
and approximate null-controllability given in \cite{GoreacMartinez2015} and
\cite{GoreacGrosuRotenstein_2016}. The main results on Riccati BSDEs are given
in Section \ref{Sec3.1}. The method allowing to deal with this stochastic
system is based on the recent ideas in \cite{CFJ_2014}. As a by-product, the
result on Riccati BSDE provides a metric-type characterization of exact
null-controllability in Section \ref{Sec3.2}. Hierachy (or absence of) between
exact null-controllability and approximate null-controllability (resp.
approximate full controllability) make the object of Section \ref{Sec3.3}
(resp. Section \ref{Sec3.4}).

\section{Model and Preliminaries}

\subsection{The Model\label{Sec2.1}}

We briefly recall the construction of a particular class of pure jump, non
explosive processes on a space $\Omega$ and taking their values in a metric
space $\left(  E,\mathcal{B}\left(  E\right)  \right)  .$ Here, $\mathcal{B}%
\left(  E\right)  $ denotes the Borel $\sigma$-field of $E.$ The elements of
the space $E$ are referred to as modes. These elements can be found in
\cite{davis_93} in the particular case of piecewise deterministic Markov
processes (see also \cite{Bremaud_1981}). To simplify the arguments, we assume
that $E$ is finite and we let $p\geq1$ be its cardinal. The process is
completely described by a couple $\left(  \lambda,Q\right)  ,$ where
$\lambda:E\longrightarrow%
%TCIMACRO{\U{211d} }%
%BeginExpansion
\mathbb{R}
%EndExpansion
_{+}$ and the measure $Q:E\longrightarrow\mathcal{P}\left(  E\right)  $, where
$\mathcal{P}\left(  E\right)  $ stands for the set of probability measures on
$\left(  E,\mathcal{B}\left(  E\right)  \right)  $ such that \ $Q\left(
\gamma,\left\{  \gamma\right\}  \right)  =0.$ Given an initial mode
$\gamma_{0}\in E,$ the first jump time satisfies $\mathbb{P}^{0,\gamma_{0}%
}\left(  T_{1}\geq t\right)  =\exp\left(  -t\lambda\left(  \gamma_{0}\right)
\right)  .$ The process $\Gamma_{t}^{\gamma_{0}}:=\gamma_{0},$ on $t<T$
$_{1}.$ The post-jump location $\gamma^{1}$ has $Q\left(  \gamma_{0}%
,\cdot\right)  $ as conditional distribution. Next, we select the inter-jump
time $T_{2}-T_{1}$ such that $\mathbb{P}^{0,\gamma_{0}}\left(  T_{2}-T_{1}\geq
t\text{ }/\text{ }T_{1},\gamma^{1}\right)  =\exp\left(  -t\lambda\left(
\gamma^{1}\right)  \right)  $ and set $\Gamma_{t}^{\gamma_{0}}:=\gamma^{1},$
if $t\in\left[  T_{1},T_{2}\right)  .$ The post-jump location $\gamma^{2}$
satisfies $\mathbb{P}^{0,\gamma_{0}}\left(  \gamma^{2}\in A\text{ }/\text{
}T_{2},T_{1},\gamma^{1}\right)  =Q\left(  \gamma^{1},A\right)  ,$ for all
Borel set $A\subset E.$ And so on. To simplify arguments on the equivalent
ordinary differential system, following \cite[Assumption (2.17)]{CFJ_2014}, we
will assume that the system stops after a non-random, fixed number $M>0$ of
jumps i.e. $\mathbb{P}^{0,\gamma_{0}}\left(  T_{M+1}=\infty\right)  =1$.

We look at the process $\Gamma^{\gamma_{0}}$ under $\mathbb{P}^{0,\gamma_{0}}$
and denote by $\mathbb{F}^{0}$ the filtration $\left(  \mathcal{F}_{\left[
0,t\right]  }:=\sigma\left\{  \Gamma_{r}^{\gamma_{0}}:r\in\left[  0,t\right]
\right\}  \right)  _{t\geq0}.$ The predictable $\sigma$-algebra will be
denoted by $\mathcal{P}^{0}$ and the progressive $\sigma$-algebra by
$Prog^{0}.$ As usual, we introduce the random measure $q$ on $\Omega
\times\left(  0,\infty\right)  \times E$ by setting $q\left(  \omega,A\right)
=\sum_{k\geq1}1_{\left(  T_{k}\left(  \omega\right)  ,\Gamma_{T_{k}\left(
\omega\right)  }^{\gamma_{0}}\left(  \omega\right)  \right)  \in A},$ for all
$\omega\in\Omega,$ $A\in\mathcal{B}\left(  0,\infty\right)  \times
\mathcal{B}\left(  E\right)  .$ The compensated martingale measure is denoted
by $\widetilde{q}$. (Further details on the compensator are given in Section
\ref{SubsTechnicalPrelim}.)

We consider a switch system given by a process $(X(t),\Gamma^{\gamma_{0}}(t))$
on the state space $%
%TCIMACRO{\U{211d} }%
%BeginExpansion
\mathbb{R}
%EndExpansion
^{N}\times E,$ for some $N\geq1$ and the family of modes $E$. $\ $The control
state space is assumed to be some Euclidian space $%
%TCIMACRO{\U{211d} }%
%BeginExpansion
\mathbb{R}
%EndExpansion
^{d},$ $d\geq1$. The component $X(t)$ follows a controlled differential system
depending on the hidden variable $\gamma$. We will deal with the following
model.%
\begin{equation}
\left\{
\begin{array}
[c]{l}%
dX_{s}^{x,u}=\left[  A\left(  \Gamma_{s}^{\gamma_{0}}\right)  X_{s}%
^{x,u}+B\left(  \Gamma_{s}^{\gamma_{0}}\right)  u_{s}\right]  ds+\int%
_{E}C\left(  \Gamma_{s-}^{\gamma_{0}},\theta\right)  X_{s-}^{x,u}%
\widetilde{q}\left(  ds,d\theta\right)  ,\text{ }s\geq0,\\
\text{ }X_{0}^{x,u}=x.
\end{array}
\right.  \label{SDE0}%
\end{equation}
The operators $A\left(  \gamma\right)  \in%
%TCIMACRO{\U{211d} }%
%BeginExpansion
\mathbb{R}
%EndExpansion
^{N\times N}$ , $B\left(  \gamma\right)  \in%
%TCIMACRO{\U{211d} }%
%BeginExpansion
\mathbb{R}
%EndExpansion
^{N\times d}$ and $C\left(  \gamma,\theta\right)  \in%
%TCIMACRO{\U{211d} }%
%BeginExpansion
\mathbb{R}
%EndExpansion
^{N\times N}$, for all $\gamma,\theta\in E$. For linear operators, we denote
by $\ker$ their kernel and by $\operatorname{Im}$ the image (or range) spaces.
Moreover, the control process $u:\Omega\times%
%TCIMACRO{\U{211d} }%
%BeginExpansion
\mathbb{R}
%EndExpansion
_{+}\longrightarrow%
%TCIMACRO{\U{211d} }%
%BeginExpansion
\mathbb{R}
%EndExpansion
^{d}$ is an $%
%TCIMACRO{\U{211d} }%
%BeginExpansion
\mathbb{R}
%EndExpansion
^{d}$-valued, $\mathbb{F}^{0}-$ progressively measurable, locally square
integrable process. The space of all such processes will be denoted by
$\mathcal{U}_{ad}$ and referred to as the family of admissible control
processes. The explicit structure of such processes can be found in
\cite[Proposition 4.2.1]{Jacobsen}, for instance. Since the control process
does not (directly) intervene in the noise term, the solution of the above
system can be explicitly computed with $\mathcal{U}_{ad}$ processes instead of
the (more usual) predictable processes.

\subsection{Technical Preliminaries\label{SubsTechnicalPrelim}}

Before giving the reduction of our backward Riccati stochastic equation to a
system of ordinary Riccati differential equations, we need to introduce some
notations making clear the stochastic structure of several concepts : final
data, predictable and c\`{a}dl\`{a}g adapted processes and compensator of the
initial random measure. The notations in this subsection follow the ordinary
differential approach from \cite{CFJ_2014}. Since we are only interested in
what happens on $\left[  0,T\right]  ,$ we introduce a cemetery state $\left(
\infty,\overline{\gamma}\right)  $ which will incorporate all the information
after $T\wedge T_{M}.$ It is clear that the conditional law of $T_{n+1}$ given
$\left(  T_{n},\Gamma_{T_{n}}^{\gamma_{0}}\right)  $ is now composed by an
exponential part on $\left[  T_{n}\wedge T,T\right]  $ and an atom at
$\infty.$ Similarly, the conditional law of $\Gamma_{T_{n+1}}^{\gamma_{0}}$
given $\left(  T_{n+1},T_{n},\Gamma_{T_{n}}^{\gamma_{0}}\right)  $ is the
Dirac mass at $\overline{\gamma}$ if $T_{n+1}=\infty$ and given by $Q$
otherwise. Finally, under the assumption $\mathbb{P}^{0,\gamma_{0}}\left(
T_{M+1}=\infty\right)  =1$, after $T_{M},$ the marked point process is
concentrated at the cemetery state.

We set $\overline{E}_{T}:\mathcal{=}\left(  \left[  0,T\right]  \times
E\right)  \cup\left\{  \left(  \infty,\overline{\gamma}\right)  \right\}  $.
For every $n\geq1,$ we let $\overline{E}_{T,n}\subset\left(  \overline{E}%
_{T}\right)  ^{n+1}$ be the set of all marks of type $e=\left(  \left(
t_{0},\gamma_{0}\right)  ,...,\left(  t_{n},\gamma_{n}\right)  \right)  ,$
where%
\begin{equation}
t_{0}=0,\text{ }\left(  t_{i}\right)  _{0\leq i\leq n}\text{ are
non-decreasing; }t_{i}<t_{i+1},\text{ if }t_{i}\leq T\text{; }\left(
t_{i},\gamma_{i}\right)  =\left(  \infty,\overline{\gamma}\right)  \text{, if
}t_{i}\leq T,\text{ }\forall0\leq i\leq n-1,\text{ } \label{Def_e_Trajectory}%
\end{equation}
and endow it with the family of all Borel sets $\mathcal{B}_{n}$. For these
sequences, the maximal time is denoted by $\left\vert e\right\vert :=t_{n}$.
Moreover, by abuse of notation, we set $\gamma_{\left\vert e\right\vert
}:=\gamma_{n}.$ Whenever $T\geq t>\left\vert e\right\vert ,$ we set
\begin{equation}
e\oplus\left(  t,\gamma\right)  :=\left(  \left(  t_{0},\gamma_{0}\right)
,...,\left(  t_{n},\gamma_{n}\right)  ,\left(  t,\gamma\right)  \right)
\in\overline{E}_{T,n+1}. \label{Def_e_Concatenation}%
\end{equation}
By defining
\begin{equation}
e_{n}:=\left(  \left(  0,\gamma_{0}\right)  ,\left(  T_{1},\Gamma_{T_{1}%
}^{\gamma_{0}}\right)  ,...,\left(  T_{n},\Gamma_{T_{n}}^{\gamma_{0}}\right)
\right)  , \label{en_RandomVar}%
\end{equation}
we get an $\overline{E}_{T,n}-$valued random variable, corresponding to our
mode trajectories.

\textbf{A c\`{a}dl\`{a}g process} $Y$ \textbf{continuous except, maybe, at
switching times} $T_{n}$ and taking its values in a topological vector space
$\mathcal{S}$ is given by the existence of a family of $\mathcal{B}_{n}%
\otimes\mathcal{B}\left(  \left[  0,T\right]  \right)  /\mathcal{B}\left(
\mathcal{S}\right)  $-measurable functions $y^{n}$ such that, for all
$e\in\overline{E}_{T,n},$ $y^{n}\left(  e,\cdot\right)  $ is continuous on
$\left[  0,T\right]  $ and constant $\left[  0,T\wedge\left\vert e\right\vert
\right]  $ and
\begin{equation}
\left.  \text{If }\left\vert e\right\vert =\infty,\text{ then }y^{n}\left(
e,\cdot\right)  =0.\text{ Otherwise, on }T_{n}\left(  \omega\right)  \leq
t<T_{n+1}\left(  \omega\right)  ,\text{ }y_{t}\left(  \omega\right)
=y^{n}\left(  e_{n}\left(  \omega\right)  ,t\right)  ,\text{ }t\leq
T\text{.}\right.  \label{Def_Cadlag}%
\end{equation}

Similar, an $\mathcal{S}-$valued $\mathbb{F}$-\textbf{predictable process} $Z$
defined on $\Omega\times\left[  0,T\right]  \times E$ is given by the
existence of a family of $\mathcal{B}_{n}\otimes\mathcal{B}\left(  \left[
0,T\right]  \right)  \otimes\mathcal{B}\left(  E\right)  /\mathcal{B}\left(
\mathcal{S}\right)  -$measurable functions $z^{n}$ satisfying%
\begin{equation}
\left.  \text{If }\left\vert e\right\vert =\infty,\text{ then }z^{n}\left(
e,\cdot,\cdot\right)  =0.\text{ On }T_{n}\left(  \omega\right)  <t\leq
T_{n+1}\left(  \omega\right)  ,\text{ }z_{t}\left(  \omega,\gamma\right)
=z^{n}\left(  e_{n}\left(  \omega\right)  ,t,\gamma\right)  ,\text{ for }t\leq
T\text{, }\gamma\in E\text{.}\right.  \label{Def_Predictable}%
\end{equation}

To deduce the form of \textbf{the compensator}, one simply writes
\[
\widehat{q}\left(  \omega,dt,d\gamma\right)  :=%
%TCIMACRO{\dsum \limits_{n\geq0}}%
%BeginExpansion
{\displaystyle\sum\limits_{n\geq0}}
%EndExpansion
\widehat{q}_{e_{n}\left(  \omega\right)  }^{n}\left(  dt,d\gamma\right)
1_{T_{n}\left(  \omega\right)  <t\leq T_{n+1}\left(  \omega\right)  \wedge T}%
\]
such that%
\begin{equation}%
\begin{array}
[c]{l}%
\text{If }n\geq M,\text{ then }\widehat{q}_{e}^{n}\left(  dt,d\gamma\right)
=\delta_{\overline{\gamma}}\left(  d\gamma\right)  \delta_{\infty}\left(
dt\right)  \smallskip\text{. }\\
\text{If }n\leq M-1,\smallskip\widehat{q}_{e}^{n}\left(  dt,d\gamma\right)
=\lambda(\gamma_{\left\vert e\right\vert })Q(\gamma_{\left\vert e\right\vert
},d\gamma)1_{\left\vert e\right\vert <\infty,t\in\left[  \left\vert
e\right\vert ,T\right]  }Leb\left(  dt\right)  +\delta_{\overline{\gamma}%
}\left(  d\gamma\right)  \delta_{\infty}\left(  dt\right)  1_{\left(
\left\vert e\right\vert <\infty,t>T\right)  \cup\left\vert e\right\vert
=\infty}.\smallskip
\end{array}
\label{Def_Compensator}%
\end{equation}

The \textbf{coefficient} function\textbf{ }$A\left(  \Gamma_{t}^{\gamma_{0}%
}\right)  $ is adapted and can be seen as follows: if $\left\vert e\right\vert
=\infty,$ then $A=0;$ otherwise, one works with $A\left(  \gamma_{\left\vert
e\right\vert }\right)  .$ Similar constructions hold true for $C.$ In fact,
the results of the present paper can be generalized to more general
path-dependence of the coefficients.

\subsection{Approximate Controllability, Exact and Approximate
Null-Controllability\label{Sec2.3}}

We will be dealing with the following notions of controllability.

\begin{definition}
(i) Given the finite time horizon $T>0$, the system (\ref{SDE0}) is said to be
approximately controllable (with initial mode $\gamma_{0}\in E)$ if, for every
final data $\xi\in\mathbb{L}^{2}\left(  \Omega,\mathcal{F}_{\left[
0,T\right]  },\mathbb{P}^{0,\gamma_{0}};%
%TCIMACRO{\U{211d} }%
%BeginExpansion
\mathbb{R}
%EndExpansion
^{N}\right)  $ (i.e. $\mathcal{F}_{\left[  0,T\right]  }$-measurable, square
integrable), every initial condition $x\in%
%TCIMACRO{\U{211d} }%
%BeginExpansion
\mathbb{R}
%EndExpansion
^{N}$ and every $\varepsilon>0$, there exists some admissible control process
$u\in\mathcal{U}_{ad}$ such that $\mathbb{E}^{0,\gamma_{0}}\left[  \left\vert
X_{T}^{x,u}-\xi\right\vert ^{2}\right]  \leq\varepsilon.$ \newline(ii) The
system (\ref{SDE0}) is said to be approximately null-controllable if the
previous condition holds for $\xi=0.$ \newline(iii) The system (\ref{SDE0}) is
said to be (exactly) null-controllable (with initial mode $\gamma_{0}\in E)$
if, for every initial condition $x\in%
%TCIMACRO{\U{211d} }%
%BeginExpansion
\mathbb{R}
%EndExpansion
^{N}$ there exists some admissible control process $u\in\mathcal{U}_{ad}$ such
that $X_{T}^{x,u}=0,$ $\mathbb{P}^{0,\gamma_{0}}$-a.s..
\end{definition}

The approach of \cite[Theorem 1]{GoreacMartinez2015} relies on the duality
between the concepts of controllability and observability. For these reasons,
one introduces the backward stochastic differential equation.
\begin{equation}
\left\{
\begin{array}
[c]{l}%
dY_{t}^{T,\xi}=\int_{E}Z_{t}^{T,\xi}\left(  \theta\right)  \widetilde{q}%
\left(  dt,d\theta\right)  -A^{\ast}\left(  \Gamma_{t}^{\gamma_{0}}\right)
Y_{t}^{T,\xi}dt-\int_{E}C^{\ast}\left(  \Gamma_{t}^{\gamma_{0}},\theta\right)
Z_{t}^{T,\xi}\left(  \theta\right)  \widehat{q}\left(  dt,d\theta\right)  ,\\
Y_{T}^{T,\xi}=\xi\in\mathbb{L}^{2}\left(  \Omega,\mathcal{F}_{\left[
0,T\right]  },\mathbb{P}^{0,\gamma_{0}};%
%TCIMACRO{\U{211d} }%
%BeginExpansion
\mathbb{R}
%EndExpansion
^{N}\right)  .
\end{array}
\right.  \label{BSDE0}%
\end{equation}
The following characterization follows from standard considerations on the
controllability linear operator(s) (cf. \cite[Theorem 1]{GoreacMartinez2015}).

\begin{theorem}
[{\cite[Theorem 1]{GoreacMartinez2015}}]\label{dualityTh}The necessary and
sufficient condition for approximate null-controllability (resp. approximate
controllability) of (\ref{SDE0}) with initial mode $\gamma_{0}\in E$ is that
any solution $\left(  Y_{t}^{T,\xi},Z_{t}^{T,\xi}\left(  \cdot\right)
\right)  $ of the dual system (\ref{BSDE0}) for which $Y_{t}^{T,\xi}\in\ker
B^{\ast}\left(  \Gamma_{t}^{\gamma_{0}}\right)  $ $,$ $\mathbb{P}%
^{0,\gamma_{0}}\mathbb{\otimes}Leb$ almost everywhere on $\Omega\times\left[
0,T\right]  $ should equally satisfy $Y_{0}^{T,\xi}=0,$ $\mathbb{P}%
^{0,\gamma_{0}}-$almost surely (resp. $Y_{t}^{T,\xi}=0,$ $\mathbb{P}%
^{0,\gamma_{0}}\mathbb{\otimes}Leb-a.s.$).
\end{theorem}

Equivalent assertions are easily obtained by interpreting the system
(\ref{BSDE0}) as a controlled, forward one :%
\begin{equation}
dY_{t}^{y,v}=\int_{E}v_{t}\left(  \theta\right)  \widetilde{q}\left(
dt,d\theta\right)  -A^{\ast}\left(  \Gamma_{t}^{\gamma_{0}}\right)
Y_{t}^{y,v}dt-\int_{E}C^{\ast}\left(  \Gamma_{t}^{\gamma_{0}},\theta\right)
v_{t}\left(  \theta\right)  \widehat{q}\left(  dt,d\theta\right)  ,\text{
}Y_{0}^{y,v}=y\in%
%TCIMACRO{\U{211d} }%
%BeginExpansion
\mathbb{R}
%EndExpansion
^{N}. \label{SDE'}%
\end{equation}
The family of admissible control processes is given by $v\in\mathcal{L}%
^{2}\left(  q;%
%TCIMACRO{\U{211d} }%
%BeginExpansion
\mathbb{R}
%EndExpansion
^{N}\right)  $ i.e. the space of all $\mathcal{P}^{0}\otimes\mathcal{B}\left(
E\right)  $ - measurable, $%
%TCIMACRO{\U{211d} }%
%BeginExpansion
\mathbb{R}
%EndExpansion
^{N}-$valued functions $v_{s}\left(  \omega,\theta\right)  $ on $\Omega\times%
%TCIMACRO{\U{211d} }%
%BeginExpansion
\mathbb{R}
%EndExpansion
_{+}\times E$ such that
\[
\mathbb{E}^{0,\gamma_{0}}\left[  \int_{0}^{T}\int_{E}\left\vert v_{s}\left(
\theta\right)  \right\vert ^{2}\widehat{q}\left(  ds,d\theta\right)  \right]
<\infty,
\]
for all $T<\infty.$

Similar duality arguments yield the following characterization of (exact) null-controllability.

\begin{proposition}
\label{Exact0CtrlProp}The necessary and sufficient condition for exact
null-controllability at time $T>0$ of (\ref{SDE0}) with initial mode
$\gamma_{0}\in E$ is the existence of a positive constant $C_{T}>0$ such that
for every initial data $y\in%
%TCIMACRO{\U{211d} }%
%BeginExpansion
\mathbb{R}
%EndExpansion
^{N}$ and every $v\in\mathcal{L}^{2}\left(  q;%
%TCIMACRO{\U{211d} }%
%BeginExpansion
\mathbb{R}
%EndExpansion
^{N}\right)  ,$ one has $\left\vert y\right\vert ^{2}\leq C_{T}\mathbb{E}%
^{0,\gamma_{0}}\left[  \int_{0}^{T}\left\vert B^{\ast}\left(  \Gamma
_{t}^{\gamma_{0}}\right)  Y_{t}^{y,v}\right\vert ^{2}dt\right]  .$
\end{proposition}

The proof is quasi-identical to the duality arguments in \cite[Theorem
1]{GoreacMartinez2015} by invoking \cite[Appendix B, Proposition
B.1]{DaPratoZabczyk1992}.

In the remaining of the section, unless stated otherwise, we assume the
control matrix $B$ to be mode-independent (constant). Using the explicit
construction of BSDE with respect to marked-point processes, an invariance
(algebraic) necessary and sufficient criterion for approximate
null-controllability has been given in \cite[Theorem 6]%
{GoreacGrosuRotenstein_2016}. We recall the following invariance concepts (cf.
\cite{Curtain_86}, \cite{Schmidt_Stern_80}).

\begin{definition}
Given a linear operator $\mathcal{A\in}%
%TCIMACRO{\U{211d} }%
%BeginExpansion
\mathbb{R}
%EndExpansion
^{N\times N}$ and a family $\mathcal{C=}\left(  \mathcal{C}_{i}\right)
_{1\leq i\leq k}\subset%
%TCIMACRO{\U{211d} }%
%BeginExpansion
\mathbb{R}
%EndExpansion
^{N\times N}$, a set $V\subset%
%TCIMACRO{\U{211d} }%
%BeginExpansion
\mathbb{R}
%EndExpansion
^{N}$ is said to be $\left(  \mathcal{A};\mathcal{C}\right)  $- invariant if
$\mathcal{A}V\subset V+%
%TCIMACRO{\tsum \limits_{i=1}^{k}}%
%BeginExpansion
{\textstyle\sum\limits_{i=1}^{k}}
%EndExpansion
\operatorname{Im}\mathcal{C}_{i}.$
\end{definition}

We construct a mode-indexed family of linear subspaces of $%
%TCIMACRO{\U{211d} }%
%BeginExpansion
\mathbb{R}
%EndExpansion
^{N}$ denoted by $\left(  V_{\gamma}^{M,n}\right)  _{0\leq n\leq M,\text{
}\gamma\in E}$ by setting
\begin{equation}
\mathcal{A}^{\ast}\left(  \gamma\right)  :=A^{\ast}\left(  \gamma\right)
-\int_{E}\left(  C^{\ast}\left(  \gamma,\theta\right)  +I\right)
\lambda(\gamma)Q(\gamma,d\theta)\text{ and }V_{\gamma}^{M,M}=\ker B^{\ast},
\label{CalA}%
\end{equation}
for all $\gamma\in E,$ and computing, for every $0\leq n\leq M-1,$
\begin{equation}
\left.  V_{\gamma}^{M,n}\text{ the largest }\left(  \mathcal{A}^{\ast}\left(
\gamma\right)  ;\left[  \left(  C^{\ast}(\gamma,\theta)+I\right)
\Pi_{V_{\theta}^{M,n+1}}:\theta\in E,\text{ }Q\left(  \gamma,\theta\right)
>0\right]  \right)  -\text{invariant subspace of }\ker B^{\ast}.\right.
\label{V^n_spaces}%
\end{equation}
Here, $\Pi_{V}$ denotes the orthogonal projection operator onto the linear
space $V\subset%
%TCIMACRO{\U{211d} }%
%BeginExpansion
\mathbb{R}
%EndExpansion
^{N}$. The explicit criterion is the following

\begin{theorem}
[{\cite[Theorem 6]{GoreacGrosuRotenstein_2016}}]\label{ThMain}The switch
system (\ref{SDE0}) is approximately null-controllable (in time $T>0$) with
$\gamma_{0}$ as initial mode, if and only if the generated set $V_{\gamma_{0}%
}^{M,0}$ reduces to $\left\{  0\right\}  .$
\end{theorem}

In the same paper \cite{GoreacGrosuRotenstein_2016}, the property of
approximate null-controllability for general systems is shown (using
convenient examples) to be strictly weaker than approximate controllability.
The following sufficient criterion is proven to guarantee the approximate controllability.

\begin{proposition}
[{\cite[Condition 10]{GoreacGrosuRotenstein_2016}}]%
\label{SuffConditionAppCtrl}Let us assume that the largest 

$\left(  \mathcal{A}^{\ast}\left(  \gamma\right)  ;\left[  \left(  C^{\ast
}(\gamma,\theta)+I\right)  \Pi_{\ker B^{\ast}}:Q\left(  \gamma,\theta\right)
>0\right]  \right)  $-invariant subspace of $\ker B^{\ast}$ is reduced to
$\left\{  0\right\}  $, for every $\gamma\in E.$ Then, for every $T>0$ and
every $\gamma_{0}\in E$, the system (\ref{SDE0}) is approximately controllable
in time $T>0.$
\end{proposition}

\section{A Backward Stochastic Riccati Equation Approach to Exact Null-Controllability}

\subsection{A Riccati Equation\label{Sec3.1}}

A simple look at \cite[Remark 4]{GoreacMartinez2015} shows that a key argument
in the analysis of controllability properties resides in a family of backward
stochastic Riccati equations. The authors of \cite[Remark 4]%
{GoreacMartinez2015} argue that their analysis is limited by solvability of
the general BSDE of the form%
\begin{equation}
\left\{
\begin{array}
[c]{l}%
dK_{t}^{\varepsilon,\mathcal{B}}=\left(  K_{t}^{\varepsilon,\mathcal{B}%
}A^{\ast}\left(  \Gamma_{t}^{\gamma_{0}}\right)  +A\left(  \Gamma_{t}%
^{\gamma_{0}}\right)  K_{t}^{\varepsilon,\mathcal{B}}-\mathcal{B}\left(
\Gamma_{t}^{\gamma_{0}}\right)  \right)  dt+\int_{E}H_{t}^{\varepsilon
,\mathcal{B}}\left(  \theta\right)  q\left(  dt,d\theta\right) \\
\text{ \ \ \ \ \ \ }+\int_{E}\left[  \left(  f_{t}^{\varepsilon,\mathcal{B}%
}\left(  \theta\right)  \right)  ^{\ast}g_{t}^{\varepsilon,\mathcal{B}}\left(
\theta\right)  f_{t}^{\varepsilon,\mathcal{B}}\left(  \theta\right)
-H_{t}^{\varepsilon,\mathcal{B}}\left(  \theta\right)  \right]  \widehat{q}%
\left(  dt,d\theta\right)  ,\\
\text{where }f_{t}^{\varepsilon,\mathcal{B}}\left(  \theta\right)  :=\left(
C\left(  \Gamma_{t}^{\gamma_{0}},\theta\right)  K_{t}^{\varepsilon
,\mathcal{B}}-H_{t}^{\varepsilon,\mathcal{B}}\left(  \theta\right)  \right)
\text{ and }g_{t}^{\varepsilon,\mathcal{B}}\left(  \theta\right)  :=\left(
\varepsilon I+K_{t}^{\varepsilon,\mathcal{B}}+H_{t}^{\varepsilon,\mathcal{B}%
}\left(  \theta\right)  \right)  ^{-1},\text{ }t\in\left[  0,T\right]  .\\
K_{t}^{\varepsilon,\mathcal{B}}=0,\text{ }\varepsilon I+K_{t}^{\varepsilon
,\mathcal{B}}+H_{t}^{\varepsilon,\mathcal{B}}\left(  \theta\right)  >0,\text{
for almost all }t\in\left[  0,T\right]  .
\end{array}
\right.  \label{RicBSDE}%
\end{equation}
Here, $\mathcal{B}\left(  \Gamma_{t}^{\gamma_{0}}\right)  $ are positive
semi-definite matrix. However, by using the structure of the jumps and
inspired by \cite{CFJ_2014}, existence of the solution of the previous BSDE
will be reduced to a family of itterated (classical) Riccati equations.

The first result gives existence and uniqueness for the solution of the
previous equation. Before stating and proving this result, let us concentrate
on the specific form of the jump contribution $H$. We consider a
c\`{a}dl\`{a}g process $K^{\varepsilon,\mathcal{B}}$ continuous except, maybe,
at switching times $T_{n}$. Then, as explained before, this can be identified
with a family $\left(  k^{n,\varepsilon,\mathcal{B}}\right)  .$ We construct,
for every $n\geq0,$%
\begin{equation}
\widehat{k}^{n+1,\varepsilon,\mathcal{B}}\left(  e,t,\gamma\right)
:=k^{n+1,\varepsilon,\mathcal{B}}\left(  e\oplus\left(  t,\gamma\right)
,t\right)  1_{\left\vert e\right\vert <t} \label{yhatn+1}%
\end{equation}
and $K_{t}^{\varepsilon,\mathcal{B}}$ can be obtained by simple integration of
the previous quantity with respect to the conditional law of $\left(
T_{n+1},\Gamma_{T_{n+1}}^{\gamma_{0}}\right)  $ knowing $\mathcal{F}_{T_{n}}.$
Then, $H$ is simply given by $h^{n,\varepsilon,\mathcal{B}}\left(
e,t,\gamma\right)  :=\widehat{k}^{n+1,\varepsilon,\mathcal{B}}\left(
e,t,\gamma\right)  -k^{n,\varepsilon,\mathcal{B}}\left(  e,t\right)  .$

The main theoretical contribution of the subsection is the following.

\begin{theorem}
\label{ThRiccati}We assume that $\mathbb{P}^{0,\gamma_{0}}\left(
T_{M+1}=\infty\right)  =1$, for some $M\geq1$. For every $\varepsilon>0$ and
every $T>0$, the Riccati BSDE (\ref{RicBSDE}) admits a unique solution
$\left(  K_{\cdot}^{\varepsilon,\mathcal{B}},H_{\cdot}^{\varepsilon
,\mathcal{B}}\left(  \cdot\right)  \right)  $ consisting of an $\mathcal{S}%
_{+}^{N}$-valued (i.e. positive semi-definite) c\`{a}dl\`{a}g process
$K_{\cdot}^{\varepsilon,\mathcal{B}}$ continuous everywhere except, maybe, at
jump times and an $\mathcal{S}^{N}$-valued (i.e. symmetric matrix-valued)
$\mathbb{F}$-predictable process $H^{\varepsilon}$.
\end{theorem}

\begin{proof}
For notation purposes, we will consider $\mathcal{B}$ to be fixed and drop the
dependency on $\mathcal{B}$. The proof consists of two steps.\newline%
\underline{Step 1.} Using the previous structure of the candidate to the
solution of the Riccati equation, one gets an equivalent system of ordinary
Riccati-type differential equations%
\[
\left\{
\begin{array}
[c]{l}%
k^{M,\varepsilon}\left(  \cdot\right)  =0,k^{n,\varepsilon}\left(
e_{n},T\right)  =0,\text{ for every }0\leq n\leq M-1,\\
dk^{n,\varepsilon}\left(  e_{n},t\right)  =\left[  k^{n,\varepsilon}\left(
e_{n},t\right)  A^{\ast}\left(  \gamma_{\left\vert e_{n}\right\vert }\right)
+A\left(  \gamma_{\left\vert e_{n}\right\vert }\right)  k^{n,\varepsilon
}\left(  e_{n},t\right)  -\mathcal{B}\left(  \gamma_{\left\vert e_{n}%
\right\vert }\right)  \right]  dt\\
+\int_{E}\left[  \left(  f^{n}\left(  e_{n},t,\theta\right)  \right)  ^{\ast
}g^{n}\left(  e_{n},t,\theta\right)  f^{n}\left(  e_{n},t,\theta\right)
\right]  \widehat{q}\left(  dt,d\theta\right)  -\int_{E}\left[  \widehat{k}%
^{n+1,\varepsilon}\left(  e_{n},t,\theta\right)  -k^{n,\varepsilon}\left(
e_{n},t\right)  \right]  \widehat{q}\left(  dt,d\theta\right)  ,\\
\text{for }t\in\left[  0,T\right]  ,\text{ where }f^{n}\left(  e_{n}%
,t,\theta\right)  =\left(  C\left(  \gamma_{\left\vert e_{n}\right\vert
},\theta\right)  +I\right)  k^{n,\varepsilon}\left(  e_{n},t\right)
-\widehat{k}^{n+1,\varepsilon}\left(  e_{n},t,\theta\right)  \text{ and }\\
g^{n}\left(  e_{n},t,\theta\right)  :=\left(  \varepsilon I+k^{n+1,\varepsilon
}\left(  e_{n}\oplus\left(  t,\gamma\right)  ,t\right)  \right)
^{-1}1_{\left\vert e\right\vert <t},\\
\text{Under the condition that }k^{n,\varepsilon}\left(  e_{n},t\right)
\geq0,\text{ for almost all }t\in\left[  0,T\right]  .
\end{array}
\right.
\]
The fact that the two systems are indeed equivalent follow from the same
arguments as those in \cite[Theorem 2]{CFJ_2014}. \underline{Step 2.} Thus,
solvability of the Riccati backward stochastic equation reduces to the
solvability of the previous system or, again, to the solvability (in
$\mathcal{S}_{+}^{N}$) of the following equation%
\[
\left.  \overset{\cdot}{p}\left(  t\right)  =p\left(  t\right)  a+a^{\ast
}p\left(  t\right)  -\Pi+p\left(  t\right)  \int_{E}\left[  b\left(
\theta\right)  r_{t}^{-1}\left(  \theta\right)  b^{\ast}\left(  \theta\right)
\right]  \nu\left(  d\theta\right)  p\left(  t\right)  ,\text{ for }%
t\in\left[  t_{0},T\right]  ,\text{ }p\left(  T\right)  =\varepsilon
I.\right.
\]
by setting, for a fixed $e_{n}$ (and $t>t_{0}:=\left\vert e_{n}\right\vert $)
\begin{align*}
a &  :=A^{\ast}\left(  \gamma_{\left\vert e_{n}\right\vert }\right)
-\lambda\left(  \gamma_{\left\vert e_{n}\right\vert }\right)  \left[
\begin{array}
[c]{c}%
\frac{1}{2}I+\int_{E}C^{\ast}\left(  \gamma_{\left\vert e_{n}\right\vert
},\theta\right)  Q\left(  \gamma_{\left\vert e_{n}\right\vert },d\theta
\right)  \\
-\varepsilon\int_{E}\left(  C^{\ast}\left(  \gamma_{\left\vert e_{n}%
\right\vert },\theta\right)  +I\right)  \left(  \varepsilon
I+k^{n+1,\varepsilon}\left(  e_{n}\oplus\left(  t,\theta\right)  ,t\right)
\right)  ^{-1}Q\left(  \gamma_{\left\vert e_{n}\right\vert },d\theta\right)
\end{array}
\right]  ,\\
\Pi &  :=\mathcal{B}\left(  \gamma_{\left\vert e_{n}\right\vert }\right)
+\varepsilon\int_{E}\left(  \varepsilon I+k^{n+1,\varepsilon}\left(
e_{n}\oplus\left(  t,\theta\right)  ,t\right)  \right)  ^{-1}%
k^{n+1,\varepsilon}\left(  e_{n}\oplus\left(  t,\theta\right)  ,t\right)
Q\left(  \gamma_{\left\vert e_{n}\right\vert },d\theta\right)  \\
b\left(  \theta\right)   &  :=C^{\ast}\left(  \gamma_{\left\vert
e_{n}\right\vert },\theta\right)  +I,\text{ }r_{t}\left(  \theta\right)
:=\left(  \varepsilon I+k^{n+1,\varepsilon}\left(  e_{n}\oplus\left(
t,\theta\right)  ,t\right)  \right)  \text{ and }\nu\left(  d\theta\right)
=\lambda\left(  \gamma_{\left\vert e_{n}\right\vert }\right)  Q\left(
\gamma_{\left\vert e_{n}\right\vert },d\theta\right)
\end{align*}
Existence and uniqueness for this equation is standard. Indeed, one notes that
$\Pi\geq0$ and $r\gg0$ (provided that $k^{n+1,\varepsilon}\geq0$)$.$ If $E$
reduces to a singletone, then this is the classical equation for deterministic
control problems (see \cite[Chapter 6, Equation 2.34]{yong_zhou_99}). The
existence and uniqueness is guaranteed by \cite[Chapter 6, Corollary
2.10]{yong_zhou_99}. For the general case, one assumes that $E$ is given by
the standard basis of $%
%TCIMACRO{\U{211d} }%
%BeginExpansion
\mathbb{R}
%EndExpansion
^{p}$ and works with
\begin{align*}
b &  =\left(  \sqrt{\nu\left(  e^{1}\right)  }b\left(  e^{1}\right)
,...,\sqrt{\nu\left(  e^{p}\right)  }b\left(  e^{p}\right)  \right)  \text{
and }\\
r_{t} &  :=\left(
\begin{array}
[c]{cccc}%
r_{t}\left(  e^{1}\right)   & 0 & ... & 0\\
0 & r_{t}\left(  e^{2}\right)   & ... & 0\\
... & ... & ... & ...\\
0 & 0 & ... & r_{t}\left(  e^{p}\right)
\end{array}
\right)  \geq\varepsilon I\gg0.
\end{align*}
The proof is complete by descending recurrence over $n\leq M.$
\end{proof}

\subsection{First Application: Null-Controllability Metric(s)\label{Sec3.2}}

\begin{proposition}
A necessary and sufficient condition for exact null-controllability of
(\ref{SDE0}) with initial mode $\gamma_{0}\in E$ at time $T>0$ is that the
pseudonorm
\begin{equation}
\left.
%TCIMACRO{\U{211d} }%
%BeginExpansion
\mathbb{R}
%EndExpansion
^{N}\ni y\longmapsto p\left(  y\right)  ,\text{ where }p^{2}\left(  y\right)
:=\inf_{v\in\mathcal{L}^{2}\left(  q;%
%TCIMACRO{\U{211d} }%
%BeginExpansion
\mathbb{R}
%EndExpansion
^{N}\right)  }\mathbb{E}^{0,\gamma_{0}}\left[  \int_{0}^{T}\left\vert
\Pi_{\left(  \ker B^{\ast}\left(  \Gamma_{t}^{\gamma_{0}}\right)  \right)
^{\bot}}\left(  Y_{t}^{y,v}\right)  \right\vert ^{2}dt\right]  \right.
\label{pMetric}%
\end{equation}
be a norm on $%
%TCIMACRO{\U{211d} }%
%BeginExpansion
\mathbb{R}
%EndExpansion
^{N}.$
\end{proposition}

\begin{proof}
It is clear that the application $p$ has non-negative values. Homogeneity is a
consequence of the equality $Y_{\cdot}^{ay,av}=aY_{\cdot}^{y,v},$ for all
$y\in%
%TCIMACRO{\U{211d} }%
%BeginExpansion
\mathbb{R}
%EndExpansion
^{N},$ all $a\in%
%TCIMACRO{\U{211d} }%
%BeginExpansion
\mathbb{R}
%EndExpansion
$ and all $v\in\mathcal{L}^{2}\left(  q;%
%TCIMACRO{\U{211d} }%
%BeginExpansion
\mathbb{R}
%EndExpansion
^{N}\right)  $ (due to the linearity of (\ref{SDE'})). To prove the
subadditivity, one simply notes $Y_{\cdot}^{y_{1}+y_{2},v^{1}+v^{2}}=Y_{\cdot
}^{y_{1},v^{1}}+Y_{\cdot}^{y_{2},v^{2}},$ for all $y^{1},y^{2}\in%
%TCIMACRO{\U{211d} }%
%BeginExpansion
\mathbb{R}
%EndExpansion
^{N}$ and all $v^{1},v^{2}\in\mathcal{L}^{2}\left(  q;%
%TCIMACRO{\U{211d} }%
%BeginExpansion
\mathbb{R}
%EndExpansion
^{N}\right)  .$ It follows that
\begin{align*}
&  p\left(  y_{1}+y_{2}\right)  \leq\left(  \mathbb{E}^{0,\gamma_{0}}\left[
\int_{0}^{T}\left\vert \Pi_{\left(  \ker B^{\ast}\left(  \Gamma_{t}%
^{\gamma_{0}}\right)  \right)  ^{\bot}}\left(  Y_{t}^{y_{1}+y_{2},v^{1}+v^{2}%
}\right)  \right\vert ^{2}dt\right]  \right)  ^{\frac{1}{2}}\\
&  \leq\left(  \mathbb{E}^{0,\gamma_{0}}\left[  \int_{0}^{T}\left\vert
\Pi_{\left(  \ker B^{\ast}\left(  \Gamma_{t}^{\gamma_{0}}\right)  \right)
^{\bot}}\left(  Y_{t}^{y_{1},v^{1}}\right)  \right\vert ^{2}dt\right]
\right)  ^{\frac{1}{2}}+\left(  \mathbb{E}^{0,\gamma_{0}}\left[  \int_{0}%
^{T}\left\vert \Pi_{\left(  \ker B^{\ast}\left(  \Gamma_{t}^{\gamma_{0}%
}\right)  \right)  ^{\bot}}\left(  Y_{t}^{y_{2},v^{2}}\right)  \right\vert
^{2}dt\right]  \right)  ^{\frac{1}{2}},
\end{align*}
for all $v^{1},v^{2}\in\mathcal{L}^{2}\left(  q;%
%TCIMACRO{\U{211d} }%
%BeginExpansion
\mathbb{R}
%EndExpansion
^{N}\right)  .$ The conclusion follows by taking infimum over such control
processes. It follows that $p$ is a pseudonorm (independently of the fact that
the system is approximately null-controllable). Necessity follows from
Proposition \ref{Exact0CtrlProp} and sufficiency from the equivalence of norms
on $%
%TCIMACRO{\U{211d} }%
%BeginExpansion
\mathbb{R}
%EndExpansion
^{N}$ by applying Proposition \ref{Exact0CtrlProp}.
\end{proof}

Using the form of the Riccati BSDE (\ref{RicBSDE}), one infers the following
explicit condition.

\begin{corollary}
\label{CorMetric0Ctrl}A necessary and sufficient condition for exact
null-controllability of (\ref{SDE0}) with initial mode $\gamma_{0}\in E$ at
time $T>0$ is that the positive-semidefinite matrix $k_{0}%
:=\underset{\varepsilon>0}{\inf}K_{0}^{\varepsilon},$ where, for every
$\varepsilon>0,$ $K^{\varepsilon}$ is the unique solution of the Riccati
equation (\ref{RicBSDE}) for $\mathcal{B}:=BB^{\ast}$ be positive definite. In
this case, the metric $p$ given in (\ref{pMetric}) is induced by $k_{0}$ i.e.%
\[
p\left(  y\right)  =\sqrt{\left\langle k_{0}y,y\right\rangle },\text{ for all
}y\in%
%TCIMACRO{\U{211d} }%
%BeginExpansion
\mathbb{R}
%EndExpansion
^{N}.
\]

\end{corollary}

\begin{proof}
This result is quite classical (see, e.g. \cite{Sarbu_Tessitore_2001} for the
Brownian-noise case). For our readers' sake, we sketch the proof. Let us fix,
for the time being $\varepsilon>0.$ Then, according to Theorem \ref{ThRiccati}%
, the Riccati equation (\ref{RicBSDE}) admits a unique solution. A simple
application of It\^{o}'s formula (cf. \cite[Chapter II, Section 5, Theorem
5.1]{Ikeda_Watanabe_1981}) to $\left\langle K_{t}^{\varepsilon}Y_{t}%
^{y,v},Y_{t}^{y,v}\right\rangle $ on $\left[  0,T\right]  $ yields%
\begin{align*}
&  \mathbb{E}^{0,\gamma_{0}}\left[  \int_{0}^{T}\left\vert \Pi_{\left[
Ker\left(  B^{\ast}\left(  \Gamma_{t}^{\gamma_{0}}\right)  \right)  \right]
^{\bot}}Y_{t}^{y,v}\right\vert ^{2}dt\right] \\
&  =\left\langle K_{0}^{\varepsilon}y,y\right\rangle -\varepsilon
\mathbb{E}^{0,\gamma_{0}}\left[  \int_{0}^{T}\left\vert u_{t}\right\vert
^{2}dt\right]  +\mathbb{E}^{0,\gamma_{0}}\left[  \int_{0}^{T}\left\vert
\begin{array}
[c]{c}%
\left(  \varepsilon I+K_{t}^{\varepsilon}+H_{t}^{\varepsilon}\left(
\theta\right)  \right)  ^{-\frac{1}{2}}f_{t}\left(  \theta\right)
Y_{t-}^{y,v}\\
-\left(  \varepsilon I+K_{t}^{\varepsilon}+H_{t}^{\varepsilon}\left(
\theta\right)  \right)  ^{\frac{1}{2}}v_{t}\left(  \theta\right)
\end{array}
\right\vert ^{2}dt\right]  .
\end{align*}
One easily notes that $\underset{v\in\mathcal{L}^{2}\left(  q;%
%TCIMACRO{\U{211d} }%
%BeginExpansion
\mathbb{R}
%EndExpansion
^{N}\right)  }{\inf}\mathbb{E}^{0,\gamma_{0}}\left[  \int_{0}^{T}\left\vert
\Pi_{\left[  Ker\left(  B^{\ast}\left(  \Gamma_{t}^{\gamma_{0}}\right)
\right)  \right]  ^{\bot}}Y_{t}^{y,v}\right\vert ^{2}dt\right]
=\underset{\varepsilon\rightarrow0}{\lim\inf}\left\langle K_{0}^{\varepsilon
}y,y\right\rangle $ and the conclusion follows.
\end{proof}

\subsection{Non-equivalence Between Exact and Approximate
Null-Controllability\label{Sec3.3}}

The following example presents a switching system which is approximately
null-controllable without being exactly null-controllable.

\begin{example}
\label{ExpNon0Ctrl}We consider a two-dimensional state space and a
one-dimensional control space. Moreover, we consider the mode to switch
randomly between three states (for simplicity, $E=\left\{  e^{1},e^{2}%
,e^{3}\right\}  $ is taken to be the standard basis of $%
%TCIMACRO{\U{211d} }%
%BeginExpansion
\mathbb{R}
%EndExpansion
^{3}$). The transition measure is given by $Q:=\left(
\begin{array}
[c]{ccc}%
0 & 1 & 0\\
0 & 0 & 1\\
0 & 1 & 0
\end{array}
\right)  .$ The coefficients are given by%
\[
\left.  A\left(  e^{1}\right)  =A:=\left(
\begin{array}
[c]{cc}%
0 & 0\\
1 & 0
\end{array}
\right)  ,\text{ }A\left(  \gamma\right)  :=0_{2\times2},\text{ if }\gamma\neq
e^{1},\text{ }B:=\left(
\begin{array}
[c]{c}%
1\\
0
\end{array}
\right)  ,\text{ }C\left(  \gamma\right)  :=0_{2\times2}.\right.
\]
\newline\textbf{a) Approximate null-controllability}\newline With the
definition (\ref{V^n_spaces}), one easily establishes $V_{\gamma}%
^{M,n}=\left\{
\begin{array}
[c]{c}%
\ker B^{\ast},\text{ if }\gamma\neq e^{1},\\
\left\{  0\right\}  \text{, if }\gamma=e^{1}.
\end{array}
\right.  ,$ for every $M\geq1$ and every $0\leq n\leq M.$ It follows that the
system is null-controllable if and only if $\gamma_{0}=e^{1}.$ \newline%
\textbf{b) Limit of the Riccati equations in the approximate null-controllable
case }(initial mode $\gamma_{0}=e^{1}$).\newline Quid est for the
controllability metric ? In this case, we recall that the limit of the
solutions of the Riccati equations is given by%
\[
p^{2}\left(  y\right)  :=\inf_{v\in\mathcal{L}^{2}\left(  q;%
%TCIMACRO{\U{211d} }%
%BeginExpansion
\mathbb{R}
%EndExpansion
^{N}\right)  }\mathbb{E}^{0,\gamma_{0}}\left[  \int_{0}^{T}\left\vert B^{\ast
}Y_{t}^{y,v}\right\vert ^{2}dt\right]  ,\text{ where }dY_{t}^{y,v}=\int%
_{E}v_{t}\left(  \theta\right)  \widetilde{q}\left(  dt,d\theta\right)
-A^{\ast}\left(  \Gamma_{t}^{\gamma_{0}}\right)  Y_{t}^{y,v}dt,\text{ }%
Y_{0}^{y,v}=y\in%
%TCIMACRO{\U{211d} }%
%BeginExpansion
\mathbb{R}
%EndExpansion
^{2}.
\]
Starting from $y:=\left(
\begin{array}
[c]{c}%
0\\
1
\end{array}
\right)  ,$ with the feedback control process $v_{t}^{\varepsilon}:=\left(
-\left(
\begin{array}
[c]{cc}%
1 & 0\\
0 & 0
\end{array}
\right)  Y_{t}+\left(
\begin{array}
[c]{c}%
0\\
\frac{e^{2\varepsilon}-1}{-e^{\varepsilon}+1+\varepsilon}1_{t\in\left[
\varepsilon,2\varepsilon\right]  }%
\end{array}
\right)  \right)  1_{t\leq T_{1}},$ one gets
\[
\left\langle Y_{t}^{y,v},\left(
\begin{array}
[c]{c}%
1\\
0
\end{array}
\right)  \right\rangle =\left[  \left(  1-e^{t}\right)  1_{t\leq\varepsilon
}+\left[  1-e^{t}+\left(  t-\varepsilon+1-e^{t-\varepsilon}\right)
\frac{e^{2\varepsilon}-1}{-e^{\varepsilon}+1+\varepsilon}\right]
1_{\varepsilon\leq t\leq2\varepsilon}\right]  1_{t\leq T_{1}}.
\]
One easily notes that $0\geq\left\langle Y_{t}^{y,v},\left(
\begin{array}
[c]{c}%
1\\
0
\end{array}
\right)  \right\rangle \geq c_{\varepsilon}:=-2\left(  e^{2\varepsilon
}-1\right)  .$ Then, by taking infimum over $\varepsilon>0$, it follows that
$p^{2}\left(
\begin{array}
[c]{c}%
0\\
1
\end{array}
\right)  =0$ and it cannot induce a norm. As consequence, by invoking
Corollary \ref{CorMetric0Ctrl}, the system fails to be (exactly) null-controllable.
\end{example}

\begin{remark}
Of course, a direct proof of null-controllability can also be given based on
the eigenvector $\left(
\begin{array}
[c]{c}%
0\\
1
\end{array}
\right)  $. Absence of null-controllability is obvious for $\gamma_{0}\neq
e^{1}$ (the system is not even approximately null-controllable). In the case
$\gamma_{0}=e^{1},$ we reason by contradiction. Let us assume that, for some
admissible control process $u$, the system is exactly controllable at time $T$
starting from $x_{0}=\left(
\begin{array}
[c]{c}%
0\\
1
\end{array}
\right)  .$ Then, prior to the first jump time,
\[
X_{t}^{x_{0},u}=\left(
\begin{array}
[c]{cc}%
1 & 0\\
t & 1
\end{array}
\right)  \left(
\begin{array}
[c]{c}%
0\\
1
\end{array}
\right)  +\int_{0}^{t}\left(
\begin{array}
[c]{cc}%
1 & 0\\
t-s & 1
\end{array}
\right)  \left(
\begin{array}
[c]{c}%
u_{s}\\
0
\end{array}
\right)  ds=\left(
\begin{array}
[c]{c}%
\int_{0}^{t}u_{s}ds\\
1+\int_{0}^{t}\left(  t-s\right)  u_{s}ds
\end{array}
\right)  .
\]
Since, on $\left[  0,T_{1}\right)  ,$ $u$ is deterministic and square
integrable, there exists $T>t_{0}>0$ such that $1+\int_{0}^{t}\left(
t-s\right)  u_{s}ds>\frac{1}{2},$ for every $t\leq t_{0}$ (consequence of the
absolute continuity)$.$ Hence, on $T_{1}\leq t_{0},$ one gets $\left\langle
X_{T_{1}-}^{x_{0},u},\left(
\begin{array}
[c]{c}%
0\\
1
\end{array}
\right)  \right\rangle >\frac{1}{2}.$ Second, one notes that $X_{T_{1}%
-}^{x_{0},u}=X_{T_{1}}^{x_{0},u}$ and $d\left\langle X_{t},\left(
\begin{array}
[c]{c}%
0\\
1
\end{array}
\right)  \right\rangle =0$ for $t\geq T_{1}$ to conclude that $\left\langle
X_{T}^{x_{0},u},\left(
\begin{array}
[c]{c}%
0\\
1
\end{array}
\right)  \right\rangle >\frac{1}{2},$ on $T_{1}\leq t_{0}.$ Since
$\mathbb{P}\left(  T_{1}\leq t_{0}\right)  =1-e^{-t_{0}}>0,$ it follows that
$u$ cannot lead to $0$ with full probability. Therefore, although it is
approximately null-controllable for some initial modes, the switch system is
never exactly null-controllable.
\end{remark}

\subsection{Exact Null-Controllability vs. Approximate (Full)
Controllability\label{Sec3.4}}

It has been shown in \cite{GoreacGrosuRotenstein_2016} that, in general,
approximate controllability is strictly stronger than approximate
null-controllability. In the light of the previous example, it is then natural
to ask oneself whether the condition on $p$ (given by (\ref{pMetric})) being a
metric implies approximate controllability of the initial system. The answer
is negative. We begin with an example of a system governed by an off/on mode
which is approximately null-controllable iff the initial mode is set off and
is never approximately controllable. We show that, for this system, the
Riccati equations give a controllability metric (iff the initial mode is set off).

\begin{example}
\label{Exp0ctrlNonAppCtrl}We consider a two-dimensional state space and a
one-dimensional control space. Moreover, we consider the mode to switch
randomly between inactive $0$ and active $1$ (i.e. $E=\left\{  0,1\right\}
$). The coefficients are given by%
\[
A\left(  \gamma\right)  =A:=\left(
\begin{array}
[c]{cc}%
0 & 0\\
1 & 0
\end{array}
\right)  ,B:=\left(
\begin{array}
[c]{c}%
1\\
0
\end{array}
\right)  ,C\left(  \gamma,1-\gamma\right)  :=\left(
\begin{array}
[c]{cc}%
-1 & 0\\
\gamma & -1
\end{array}
\right)  .
\]
\newline\textbf{a) Approximate null-controllability}\newline One easily
establishes $V_{\gamma}^{M,n}=span\left\{  \gamma e_{2}\right\}  ,$ for every
$M\geq1$ and every $0\leq n\leq M.$ It follows that the system is
null-controllable if and only if $\gamma_{0}=0.$ \newline\textbf{b)
Approximate controllability }(initial mode $\gamma_{0}=0$)\newline One easily
checks that $Y_{t}:=\left(
\begin{array}
[c]{c}%
0\\
\Gamma_{t}^{0}%
\end{array}
\right)  ,Z_{t}\left(  \cdot\right)  =\left(
\begin{array}
[c]{c}%
0\\
\left(  -1\right)  ^{\Gamma_{t-}^{0}}%
\end{array}
\right)  $, for $0\leq t\leq T$ satisfies the BSDE (\ref{BSDE0}) with final
data $Y_{T}=\left(
\begin{array}
[c]{c}%
0\\
\Gamma_{T}^{0}%
\end{array}
\right)  .$ Since this solution stays in $\ker B^{\ast}$ and it is not
trivially zero, it follows that the system is (never) approximately
controllable. \newline\textbf{c) Riccati equations in the approximate
null-controllable case }(initial mode $\gamma_{0}=0$)\newline One easily notes
that the Riccati equations lead to
\begin{align*}
dk^{0,\varepsilon}\left(  0,t\right)  =  &  \left[  k^{0,\varepsilon}\left(
0,t\right)  \left(
\begin{array}
[c]{cc}%
0 & 1\\
0 & 0
\end{array}
\right)  +\left(
\begin{array}
[c]{cc}%
0 & 0\\
1 & 0
\end{array}
\right)  k^{0,\varepsilon}\left(  0,t\right)  -\left(
\begin{array}
[c]{cc}%
1 & 0\\
0 & 0
\end{array}
\right)  +k^{0,\varepsilon}\left(  0,t\right)  \right]  dt\\
+  &  \left[  \widehat{k}^{1,\varepsilon}\left(  0,0,t,1\right)  \left(
\varepsilon I+\widehat{k}^{1,\varepsilon}\left(  0,0,t,1\right)  \right)
^{-1}\widehat{k}^{1,\varepsilon}\left(  0,0,t,1\right)  -\widehat{k}%
^{1,\varepsilon}\left(  0,0,t,1\right)  \right]  dt,t\in\left[  0,T\right]  .
\end{align*}
Since $\widehat{k}^{n+1,\varepsilon}\left(  0,0,t,1\right)  \geq0,$ this
solution is at least equal to the one given by $\overline{K}_{t}:=\left(
\begin{array}
[c]{cc}%
a\left(  t\right)  & b\left(  t\right) \\
b\left(  t\right)  & c\left(  t\right)
\end{array}
\right)  ,$ where $a\left(  t\right)  :=1-e^{t-T},$ $b\left(  t\right)
:=\left(  T+1-t\right)  e^{t-T}-1$ and $c\left(  t\right)  :=2-\left(
1+\left(  T+1-t\right)  ^{2}\right)  e^{t-T}$. Hence, $\overline{K}_{0}$ is
positive definite (for every $T>0$) and so is $\underset{\varepsilon
\rightarrow0+}{\lim\inf}K_{0}^{0,\varepsilon}$ (to prove this, one simply
studies the sign of the function $T\mapsto\left(  1-e^{-T}\right)  \left[
2-\left(  2+T^{2}+2T\right)  e^{-T}\right]  -\left[  \left(  T+1\right)
e^{-T}-1\right]  ^{2}$ on $%
%TCIMACRO{\U{211d} }%
%BeginExpansion
\mathbb{R}
%EndExpansion
_{+}^{\ast}$).
\end{example}

\begin{remark}
The reader may want to note that in the case presented in the previous
example, the system is approximately null-controllable if and only if it is
(exactly) null-controllable. To give an explicit control leading from $x_{0}$
to $0$ for the initial setting $\gamma_{0}=0,$ one proceeds as follows. Prior
to the first jump, the system is a deterministic one and given by
\[
d\Phi_{t}^{x_{0},u^{1}}=\left(  \left(  A-C\left(  0,1\right)  \right)
\Phi_{t}^{x_{0},u^{1}}+Bu_{t}^{1}\right)  dt,\text{ }\Phi_{0}^{x_{0},u^{1}%
}=x_{0}\in%
%TCIMACRO{\U{211d} }%
%BeginExpansion
\mathbb{R}
%EndExpansion
^{2}.
\]
Since Kalman's condition is satisfied for this dterministic system, it is
exactly null-controllable at time $T>0.$ An explicit control is obtained by
considering the (deterministic) controllability Gramian $G_{T}$, respectively
the induced (open-loop) control process $u^{1}$ given by
\[
G_{T}:=\int_{0}^{T}e^{\left(  A-C\left(  0,1\right)  \right)  \left(
T-s\right)  }BB^{\ast}e^{\left(  A-C\left(  0,1\right)  \right)  ^{\ast
}\left(  T-s\right)  }ds,\text{ }u^{1}\left(  t\right)  :=-B^{\ast}e^{\left(
A-C\left(  0,1\right)  \right)  ^{\ast}\left(  T-t\right)  }G_{T}%
^{-1}e^{\left(  A-C\left(  0,1\right)  \right)  ^{\ast}T}x_{0}1_{t\leq T}.
\]
We obtain a stochastic control by setting $u^{p}\left(  t\right)  =0,$ i.e. we
take null-control after the first jumping time. Then, it is obvious that, on
$T_{1}\geq T,$ $X_{T}^{x_{0},u}=\Phi_{T}^{x_{0},u^{1}}=0$. On $T_{1}<T,$
$X_{T_{1}}^{x_{0},u}=0$ and the conclusion follows.
\end{remark}

Finally, one is entitled to ask whether approximate null-controllability
implies exact null-controllability. The answer is, again, negative proving
that approximate controllability and exact null-controllability are, in
general, completely different properties. To illustrate this, let us take,
once again, a glance at the first example.

\begin{example}
We consider a two-dimensional state space and a one-dimensional control space.
Moreover, we consider the mode to switch randomly between three states (for
simplicity, $E=\left\{  e_{1},e_{2},e_{3}\right\}  $ is taken to be the
standard basis of $%
%TCIMACRO{\U{211d} }%
%BeginExpansion
\mathbb{R}
%EndExpansion
^{3}$). The transition measure is given by $Q:=\left(
\begin{array}
[c]{ccc}%
0 & 1 & 0\\
0 & 0 & 1\\
0 & 1 & 0
\end{array}
\right)  .$ The coefficients are given by%
\[
\left.  A\left(  e_{1}\right)  =A:=\left(
\begin{array}
[c]{cc}%
0 & 0\\
1 & 0
\end{array}
\right)  ,\text{ }A\left(  \gamma\right)  :=0_{2\times2},\text{ if }\gamma\neq
e_{1},\text{ }B:=\left(
\begin{array}
[c]{c}%
1\\
0
\end{array}
\right)  ,\text{ }C\left(  \gamma\right)  :=0_{2\times2}\right.
\]
As we have seen before (in Example \ref{ExpNon0Ctrl}), this system is never
exactly null-controllable. We assume the system to only jump once. We consider
a solution of%
\[
dY_{t}^{y,v}=\int_{E}v_{t}\left(  \theta\right)  \widetilde{q}\left(
dt,d\theta\right)  -A^{\ast}\left(  \Gamma_{t}^{\gamma_{0}}\right)
Y_{t}^{y,v}dt,Y_{0}^{y,v}=y\in%
%TCIMACRO{\U{211d} }%
%BeginExpansion
\mathbb{R}
%EndExpansion
^{N}%
\]
that belongs to $\ker B^{\ast},$ $\mathbb{P}^{0,e_{1}}-a.s.$ Due to the
approximate null-controllability, it follows that $y=0.$ Prior to the first
jump, $v$ is given by a deterministic function $v^{1}=\left(
\begin{array}
[c]{c}%
v^{1,1}\\
v^{1,2}%
\end{array}
\right)  $ and $Y_{t}^{y,v}$ coincides with the deterministic solution of
$d\Phi_{t}^{v^{1}}=\left(  -v_{t}^{1}-A^{\ast}\Phi_{t}^{v^{1}}\right)  dt,$
$\Phi_{0}^{v^{1}}=0.$ Since $\Phi_{t}^{v^{1}}\in\ker B^{\ast}=span\left\{
\left(
\begin{array}
[c]{c}%
0\\
1
\end{array}
\right)  \right\}  $ (for all $t\in\left[  0,T\right]  $) one has, for
$Leb-$almost all $t\in\left[  0,T\right]  ,$ $-v_{t}^{1,1}=\left\langle
\Phi_{t}^{v^{1}},\left(
\begin{array}
[c]{c}%
0\\
1
\end{array}
\right)  \right\rangle .$ Hence, at the first jumping time, one has
\[
0=\left\langle Y_{T_{1}}^{y,v},\left(
\begin{array}
[c]{c}%
1\\
0
\end{array}
\right)  \right\rangle =\left\langle \Phi_{T_{1}}^{v},\left(
\begin{array}
[c]{c}%
1\\
0
\end{array}
\right)  \right\rangle -\left\langle \Phi_{T_{1}}^{v^{1}},\left(
\begin{array}
[c]{c}%
0\\
1
\end{array}
\right)  \right\rangle =-\left\langle \Phi_{T_{1}}^{v^{1}},\left(
\begin{array}
[c]{c}%
0\\
1
\end{array}
\right)  \right\rangle .
\]
One deduces that $\left\langle \Phi_{t}^{v^{1}},\left(
\begin{array}
[c]{c}%
0\\
1
\end{array}
\right)  \right\rangle =0$ and $v_{t}^{1,1}=0,$ $Leb$-almost surely on
$\left[  0,T\right]  .$ Then $t$he derivative of $\left\langle \Phi_{t}%
^{v^{1}},\left(
\begin{array}
[c]{c}%
0\\
1
\end{array}
\right)  \right\rangle $ is null i.e. $v_{t}^{1,2}=0,$ $Leb$-almost surely on
$\left[  0,T\right]  $. As a consequence, for all (due to right continuity)
$t\in\left[  0,T_{1}\right]  ,$ $Y_{t}^{y,v}=0,$ $\mathbb{P}^{0,e_{1}}$-almost
surely$.$ Since the process is no longer allowed to jump after $T_{1},$ it
follows that the equality actually holds on $\left[  0,T\right]  $ and the
initial system (\ref{SDE0}) is approximately controllable (cf. Theorem
\ref{dualityTh}).
\end{example}

\bibliographystyle{plain}
\bibliography{bibliografie_17042016}

\end{document}